\documentclass[oneside,a4paper,12pt,notitlepage]{article}
\usepackage[left=0.8in, right=0.8in, top=1.5in, bottom=1.5in]{geometry}
\usepackage[T1]{fontenc} 
\usepackage[utf8]{inputenc}
\usepackage{lipsum}
\usepackage{lmodern}
\usepackage{amssymb}
\usepackage{amsthm}
\usepackage{bm}
\usepackage{mathtools}
\usepackage{braket}
\usepackage{esint}
\usepackage{todonotes}

\newcommand{\abs}[1]{{\left|#1\right|}}
\newcommand{\ov}[1]{{\overline{#1}}}
\newcommand{\norma}[1]{{\left\Vert#1\right\Vert}}

\newcommand{\F}{\mathcal{F}}
\newcommand{\eps}{\varepsilon}

\newcommand{\parzder}[2]{{\frac{\partial #1}{\partial #2}}}

\newcommand{\measurerestr}{%
  \,\raisebox{-.127ex}{\reflectbox{\rotatebox[origin=br]{-90}{$\lnot$}}}\,%
}

\usepackage{booktabs}
\usepackage{graphicx}
\usepackage{tikz}
\usetikzlibrary{patterns}
\usepackage{multicol}
\usepackage{caption}
\usepackage{enumerate}
\usepackage[skins,theorems]{tcolorbox}
\tcbset{highlight math style={enhanced,
		colframe=black,colback=white,arc=0pt,boxrule=1pt}}
\def\XXint#1#2#3{{\setbox0=\hbox{$#1{#2#3}{\int}$}
    \vcenter{\hbox{$#2#3$}}\kern-.5\wd0}}

\theoremstyle{definition}
\newtheorem{definizione}{Definition}[section]
\theoremstyle{plain}
\newtheorem{teorema}{Theorem}[section]

\newtheorem{lemma}[teorema]{Lemma}
\newtheorem{prop}[teorema]{Proposition}
\newtheorem{corollario}[teorema]{Corollary}
\theoremstyle{definition}
\newtheorem{esempio}{Example}[section]
\newtheorem{oss}[esempio]{Remark}

\DeclareMathOperator{\R}{\mathbb{R}}

\DeclareMathOperator{\supp}{\text{supp}}

\makeatletter
\newcommand{\myfootnote}[2]{\begingroup
	\def\@makefnmark{}%
	\addtocounter{footnote}{-1}%
	\footnote{\textbf{#1} #2}%
	\endgroup}
\makeatother
\newcommand{\BV}{\text{BV}}

\numberwithin{equation}{section}

\usepackage[style = alphabetic, maxbibnames=99, maxcitenames = 99, maxbibnames = 99, maxalphanames = 99]{biblatex}
\usepackage{hyperref}
\hypersetup{linktoc=none, bookmarksnumbered, colorlinks=true, linkcolor=red}

\title{On the gradient rearrangement of functions}

\author{Vincenzo Amato, Andrea Gentile, Carlo Nitsch, Cristina Trombetti}
\date{}

\begin{document}

\maketitle

\begin{abstract}

In this paper, we introduce a symmetrization technique for the gradient of a $\BV$ function, which separates its absolutely continuous part from its singular part (sum of the jump and the Cantorian part). In particular, we prove an $\text{\emph{L}}^{\text{1}}$ comparison between the function and its symmetrized. 
     
Furthermore, we apply this result to obtain Saint-Venant type inequalities for some geometric functionals.

    

\end{abstract}

\maketitle

\section{Introduction}
\nocite{Alt_Caffarelli,ALT_optimization_problem_with_prescribed_rearrangement,Alvino_Trombetti_costanti_maggiorazione,AG_rearrangement_gradient_Sobolev,ADM,Ambrosio_Fusco_Pallara,BdPNT,Bucur_Buttazzo_Nitsch,Cianchi_Lq_norm,CianchiFusco_BV,De_Giorgi_Carriero_Leaci,Torsional_Rigidity_Diaz_Weinstein,Evans_Gariepy,Ferone_Posteraro,Ferone_Posteraro_Volpicelli,Fleming_Rishel,GN,book_kahowl,kesavan,Mercaldo_remark_comparison_Hamilton_Jacobi,Polya_Szego,Torsional_rigidity_Polya_Weinstein,Talenti:compendio}
Let $\Omega$ be a bounded open set  of $\R^n$ with finite perimeter (see section \ref{section_preliminary} for its definition) and let us denote, as in \cite{BdPNT}, by
\[
    \text{BV}_0(\Omega):= \left\{u \in \text{BV}(\R^n) :\, u\equiv 0 \text{ in }\R^n \setminus \Omega \right\}.
\]

The aim of the present paper is to define a  symmetrization of the distributional gradient of a $BV$ function.

The interest in this topic essentially derives from the work \cite{GN} where the authors deal with the following problems involving Hamilton-Jacobi equation

\begin{subequations}\label{HJ}
    \noindent\begin{minipage}{0.48\textwidth}
\begin{equation}
\begin{cases}
                    H(\nabla u) = f & \text{a.e. in } \Omega \\
                    u = 0 & \text{on } \partial \Omega
                \end{cases}
                \label{HJ_u}
\end{equation}
    \end{minipage}
    \begin{minipage}{0.52\textwidth}
\begin{equation}
 \label{HJ_v}
 \quad
                \begin{cases}
                    K(\abs{\nabla v}) = f_{\sharp} & \text{a.e. in } \Omega^{\sharp} \\
                    v = 0 & \text{on } \partial \Omega^{\sharp}
                \end{cases}
\end{equation}
    \end{minipage}\vskip1em
\end{subequations}
where $\Omega ^{\sharp }$ is the ball centered at the origin with the same measure as $\Omega$ (in the sequel just centered ball), $H: \R^n \to \R$ and $K:\R \to \R$ are measurable functions,  $u,v \in W^{1,p}_0$ and $f_{\sharp}$ is the increasing rearrangement of $f$ (see Section \ref{section_preliminary} for its definition). 

In particular, under suitable assumptions on $H$ and $K$, it is proven (\cite[Theorem 2.2]{GN}) that whenever $u,v$ are solutions to \eqref{HJ_u} and \eqref{HJ_v} respectively, then
\begin{equation*}
    \norma{u}_{L^1(\Omega)} \leq \norma{v}_{L^1(\Omega^{\sharp})}.
\end{equation*}

In \cite{ALT_optimization_problem_with_prescribed_rearrangement} the authors study the problem of maximization of the $L^q$ norm among functions with prescribed gradient rearrangement. 
Precisely, the following cases are considered
\begin{itemize}
        \item $1 \leq q \leq \frac{np}{n-p}$ if $p<n$,
        
        \item $1 \leq q <+\infty$ if $p = n$,
        
        \item $1 \leq q \leq + \infty$ if $p>n$,
        
    \end{itemize}
and for a fixed $\varphi = \varphi^{*} \in L^p(0,\abs{\Omega})$, they define 
\begin{equation*}
        I(\Omega) : =
        \sup \left \{ \norma{v}_{L^q} \, : \; \;
         \begin{aligned}
            & \abs{\nabla v} \leq f \; \text{a.e. in }\Omega, \\
            & v \in W_0^{1,p}(\Omega) \\
            & f \geq 0, \, f^* = \varphi^*
        \end{aligned}
        \right\},
\end{equation*}
and they proved the following
\begin{teorema}{\cite[Theorem 3.1]{ALT_optimization_problem_with_prescribed_rearrangement}}
    \label{3.1alt}
    Let $\Omega$ be a bounded open set in $\R^n$,  let $\Omega^{\sharp}$ be the centered ball, let $R$ be its radius and let $p,q,\varphi$ be as defined above.
    
    Then, there exist $v, g$ spherically symmetric on $\Omega^{\sharp}$ such that $g^* = \varphi$, $I(\Omega^{\sharp}) = \norma{v}_{L^q}$, and thus
    \[
        v \in W_0^{1,p}(\Omega), \, v \geq 0, \, \abs{\nabla v} = g \qquad \text{a.e. in }\Omega^{\sharp}.
    \]
    Furthermore $I(\Omega^{\sharp}) \geq I(\Omega)$ for all open sets $\Omega$ in $\R^n$ with $\lvert \Omega^{\sharp} \rvert = \abs{\Omega}$.
\end{teorema}

In \cite{Cianchi_Lq_norm} the author proved a representation formula for the function $g$, the existence of which was proved in Theorem \ref{3.1alt}.

Let us also mention that in \cite{Ferone_Posteraro, Ferone_Posteraro_Volpicelli} the authors studied the optimization of the norm of a Sobolev function in the class of functions with prescribed rearrangement of the gradient.

The case of a Sobolev non-zero trace function for $q=1$ is instead studied in \cite{AG_rearrangement_gradient_Sobolev}.

The literature concerning rearrangements in the spaces $W^{1,p}$ is exhaustive, whereas, to our knowledge, results on $\text{BV}$ functions are rarer. One of the most relevant papers in this framework is \cite{CianchiFusco_BV} where the authors extend the validity of Polya-Szeg\"o inequality to $\text{BV}$ functions. More specifically, they proved that if $u \in \text{BV}(\R^n)$, then its Schwarz rearrangement $u^{\sharp}$ (see Section \ref{section_preliminary} for its definition) belongs to $\text{BV}(\R^n)$ and it holds \cite[Theorem 1.3]{CianchiFusco_BV}
\begin{equation}
    \label{polya_szego_cianchi_fusco}
    \begin{split}
        \lvert D u^{\sharp} \rvert (\R^n) &\leq \lvert D u \rvert (\R^n) \\
        \lvert D^{\mathrm{s}} u^{\sharp} \rvert (\R^n) & \leq \lvert D^{\mathrm{s}} u \rvert (\R^n) \\
        \lvert D^{\mathrm{j}} u^{\sharp} \rvert (\R^n) & \leq \lvert D^{\mathrm{j}} u \rvert (\R^n)
    \end{split}
\end{equation}
where $D^s$ and $D^j$ denote respectively the singular and the jump part of the gradient (see \cite{CianchiFusco_BV} for their definitions). There is no analogue of \eqref{polya_szego_cianchi_fusco} for the absolutely continuous and the Cantorian part of the gradient, i.e. in the symmetrization procedure the total variation of $D^a$ and $D^c$ can be increased, as shown in the example given in \cite{CianchiFusco_BV}.

In this paper, we want to introduce a symmetrization that keeps the absolutely continuous part separate from the singular part (sum of jump and Cantorian part) of the gradient. To be more precise, we define the radial function $u^{\star} \in W^{1,1}(\Omega^{\sharp})\cap \text{BV}_0(\Omega^{\sharp}) \cap L^{\infty}(\Omega^\sharp)$ such that
  \begin{equation}
        \label{eq_che_risolve_u_picche}
        \begin{cases}
            \lvert \nabla u^{\star} \rvert(x) = \abs{\nabla^a u}_{\sharp}(x) & \text{a.e. in }\Omega^{\sharp} \\[1ex]
            u^\star (x)=  \cfrac {1}{\displaystyle{ \text{Per}(\Omega^\sharp)}} \lvert D^s u \rvert (\R^n)   &\text{ on } \partial \Omega^{\sharp}
        \end{cases}
        ,
    \end{equation}
where $\nabla^a u$ and $D^s u$ will be defined in Section \ref{section_preliminary}.

The main theorem can be stated as follows.
\begin{teorema}
    \label{Teorema_che_scriveremo}
    Let $\Omega \subset \R^n$ be a bounded open set with finite perimeter and let $\Omega^\sharp$ be the centered ball. Assume that $u$ is a non-negative function belonging to $ \text{BV}_0(\Omega)$  and assume that $u^\star$ is defined as in \eqref{eq_che_risolve_u_picche}, then
    \begin{equation*}
        \norma{u}_{L^1(\Omega)} \leq \norma{u^{\star}}_{L^1(\Omega^{\sharp})}.
    \end{equation*}
\end{teorema}

We will also deal with some applications, in particular we will consider
\begin{itemize}
    
    \item a penalized torsional rigidity problem
    \begin{equation*}
        T_\mathcal{F}(\Omega,\Lambda) := - \inf_{\psi \in H_0^1(\Omega)} \biggl( \frac{1}{2} \int_{\Omega} \abs{\nabla \psi}^2 \, dx - \int_{\Omega} \abs{\psi} \, dx + \Lambda\abs{ \left\{ \abs{\nabla \psi} \neq 0 \right\}} \biggr);
    \end{equation*}

    \item a modified torsional rigidity
    \begin{equation*}
        \frac{1}{T_\mathcal{G}(\Omega,m)} := \inf_{\psi \in H^1(\Omega)}  \frac{\displaystyle{\int_{\Omega} \abs{\nabla \psi}^2 \, dx + \frac{1}{m}\left( \int_{\partial \Omega} \abs{\psi} \, d\mathcal{H}^{n-1}\right)^2}}{\displaystyle{\left(\int_{\Omega} \abs{\psi} \, dx\right)^2}}.
    \end{equation*}
\end{itemize}
In both cases, we will prove a Saint-Venant type inequality:
\begin{equation*}
    T_{\mathcal{F}}(\Omega,\Lambda) \leq T_{\mathcal{F}}(\Omega^{\sharp}, \Lambda), \qquad	T_{\mathcal{G}}(\Omega,m) \leq T_{\mathcal{G}}(\Omega^{\sharp},m).
\end{equation*}

The paper is organized as follows: in Section \ref{section_preliminary} we recall some preliminary results and useful tools for our aim, in Section \ref{sec_proof_of_main_results} we prove our main result on the symmetrization of the gradient for a BV function, while in Section \ref{sec_some_applications} we present some applications of this kind of symmetrization.

\section{Notations and preliminaries}
\label{section_preliminary}

\subsection{Functions of bounded variation}
Let us summarize some basic notions concerning $\text{BV}$ functions, for all the  details we refer for instance to \cite{Ambrosio_Fusco_Pallara, CianchiFusco_BV, Evans_Gariepy}.

In the following, $\Omega$ will be an open set of $\R^n$.

\begin{definizione}
    A function $u \in L^1(\Omega)$ is said to be a \textbf{function of bounded variation} in $\Omega$ if its distributional derivative is a Radon measure, i.e.
    \begin{equation*}
        \int_{\Omega} u \parzder{\varphi}{x_i} \, dx = \int_{\Omega} \varphi \, dD^i u \qquad \forall \varphi \in C_C^{\infty}(\Omega),
    \end{equation*}
    with $Du$ a $\R^n$-valued measure in $\Omega$. The total variation of $Du$ will be denoted with $\abs{Du}$.

    The set of functions of bounded variation in $\Omega$ is denoted by $\text{BV}(\Omega)$ and it is a Banach space with respect to the norm $\norma{u}_{\text{BV}(\Omega)} : = \norma{u}_{L^1(\Omega)} + \abs{Du}(\Omega)$.
\end{definizione}

\begin{definizione}
    Let $E$ be a $\mathcal{L}^n$-measurable set. The \textbf{perimeter} of $E$ inside $\Omega$ is defined as
    \[
        \text{Per}(E,\Omega) : = \abs{D \chi_E}(\Omega),
    \]
    and we say that $E$ is a \textbf{set of finite perimeter} in $\Omega$ if $\chi_E \in \text{BV}(\Omega)$. If $\Omega = \R^n$, we denote $\text{Per}(E):=\text{Per}(E,\R^n)$.
\end{definizione}

It is also worth mentioning the isoperimetric inequality for sets of finite perimeter.

\begin{teorema}[Isoperimetric inequality]
    Let $E\subset\R^n$ be a bounded set of finite measure. Then it holds
    \begin{equation*}
        \abs{E} \leq n^{-\frac{n}{n-1}}\omega_n^{-\frac{1}{n-1}}\left[\text{Per}(E)\right]^{\frac{n}{n-1}},
    \end{equation*}
    where $\omega_n$ is the measure of $n$-dimensional ball of radius $1$.
\end{teorema}

By the Lebesgue decomposition Theorem, each component of $Du$  can be decomposed with respect to the Lebesgue measure, namely
\[
    D_i u = D_i^{\mathrm{a}}u + D_i^{\mathrm{s}}u \qquad \text{ with } D_i^{\mathrm{a}}u \ll \mathcal{L}^n, \quad D_i^{\mathrm{s}}u \perp \mathcal{L}^n.
\]
and
\[
    D_i^{\mathrm{a}}u = f_i \measurerestr \mathcal{L}^n,
\]
for some $f_i \in L^1(\Omega)$. So, defining
\[  
    \frac{\partial u}{\partial x_i} : = f_i, \qquad 
    \nabla^{\mathrm{a}} u = \biggl(  \frac{\partial u}{\partial x_1},\cdots, \frac{\partial u}{\partial x_n} \biggr) \qquad \text{ and } D^{\mathrm{s}} u = \bigl( D_1^{\mathrm{s}}u, \ldots, D_n^{\mathrm{s}}u  \bigr),
\]
we can write
\begin{equation*}
    dDu = \nabla^{\mathrm{a}} u \measurerestr  \mathcal{L}^n + dD^{\mathrm{s}}u.
\end{equation*}
Clearly it holds
\[
    \abs{Du}(A) = \lvert D^{\mathrm{a}} u \rvert (A) + \lvert D^{\mathrm{s}} u \rvert (A) = \int_{A} \abs{\nabla^{\mathrm{a}} u} \, dx + \lvert D^{\mathrm{s}} u \rvert (A),
\]
for every Borel set $A \subseteq \Omega$.

 Let us recall the following \textbf{Fleming-Rishel formula} (see \cite{Fleming_Rishel} or \cite{Evans_Gariepy}):
\begin{teorema}[Fleming-Rishel formula]
    \label{fleming_rishel_teo}
    Let $\Omega\subset \R^n $ be an open  set  and  let $u \in \text{BV}(\Omega)$, then for almost every $t \in (-\infty,+\infty)$ the set $
        \{ u>t\}
$
    has finite perimeter in $\Omega$ and it holds
    \begin{equation}
        \label{Fleming_Rishel_inequality}
        \abs{Du}(\Omega) = \int_{-\infty}^{+\infty} \text{Per}(\{ u>t\}, \Omega) \, dt.
    \end{equation}
    Moreover if $u \in L^1(\Omega)$ and
    \[
        \int_{-\infty}^{+\infty} \text{Per} (\{ u > t \}, \Omega) \, dt < +\infty,
    \]
    then $u \in \text{BV}(\Omega)$ and consequently \eqref{Fleming_Rishel_inequality} holds.
\end{teorema}

\subsection{Rearrangements of functions}
We now briefly recall some notions about rearrangements. We refer for instance to \cite{Talenti:compendio,kesavan,book_kahowl} for all the details.

\begin{definizione}
    \noindent Let $\Omega$ be a measurable set and let $u \colon \Omega \to \R$ be a measurable function, the \textbf{distribution function} of $u$ is defined as
    \begin{equation*}
        \mu \colon [0,+\infty) \to [0,+\infty) \qquad \mu(t) = \left \lvert \bigl( \left\{x \in \Omega \,  : \abs{u(x)}>t\right\} \bigr) \right \rvert 
    \end{equation*}
    where, here and throughout the paper, $\abs{E}$ denotes the $n$-dimensional Lebesgue measure of a measurable set $E$.
\end{definizione}

\noindent It can be proved that
\begin{itemize}
	\item $\mu$ is a decreasing function in $[0,+\infty)$;
	
	\item $\mu$ is right-continuous;
	
	\item $\mu(0) = \abs{ \supp u }$ and $\mu(+\infty) = 0$;

        \item $\mu(t^-) = \bigl \lvert \left\{ x \in \Omega \, : \, \abs{u(x)} \geq t \right\} \bigr \rvert $.

\end{itemize}

\begin{definizione}
    Let $u \colon \Omega \to \R$ be a measurable function, the \textbf{decreasing rearrangement} of $u$ is defined as
    \begin{equation*}
        u^* \colon \R^+ \to \R^+ \qquad u^*(s) = \inf\left\{t>0 \, : \mu(t) \leq s\right\}
    \end{equation*}
    and the \textbf{increasing rearrangement} of $u$ as
    \begin{equation*}
        u_* \colon [0,\abs{\Omega}] \to \R^+ \qquad u_*(s) = u^*(\abs{\Omega}-s)
    \end{equation*}
\end{definizione}
\noindent It can be proved that
\begin{itemize}
    \item $u^*$ ($u_*$) is a decreasing (increasing) function in $[0,+\infty)$;
    
    \item $u^*$ and $u_*$ are lower semi-continuous;
    
    \item whenever $u \in L^{\infty}(\Omega)$ $u^*(0) = \norma{u}_{L^{\infty}(\Omega)}$ and $u^*(t) = 0$ $\forall t \geq \abs{ \supp u }$;
    
    \item $u_*(\abs{\Omega}) = \norma{u}_{L^{\infty}(\Omega)}$ and $u_*(t) = 0$ $\forall 0 \leq t  \leq \abs{\Omega} -\abs{ \supp u }$;

    \item $u^*$ and $u_*$ have the same distribution function as $u$, so by Cavalieri's principle the $L^p$ norms are equal for every $p$;
    
    \item $u^\ast (\mu(t)) \leq t$ for every non-negative $t$, $\mu (u^\ast(s)) \leq s$ for every non-negative $s$;
       \item $u^\ast (\mu(t)^-) \geq t$ for every non-negative $t$, $\mu (u^\ast(s)^-) \geq s$ for every non-negative $s$;

    \item  the \textbf{Hardy-Littlewood inequality:} for any $u,v \colon \Omega \subseteq \R^n \to \R$
    \begin{equation}
        \label{Hardy-Littlewood}
        \int_{\Omega} \abs{u(x)v(x)} \, dx \leq \int_{\Omega^{\sharp}} u^*(x) v^*(x) \, dx = \int_{\Omega} u_*(x) v_*(x) \, dx
    \end{equation}

\end{itemize}

\begin{definizione}
    Let $u \colon \Omega \to \R$ be a measurable function. The \textbf{Schwarz rearrangement} or the \textbf{spherically symmetric decreasing rearrangement} of $u$ is defined as
    \begin{equation*}
        u^{\sharp} \colon \R^n \to \R^+ \qquad u^{\sharp}(x) = u^*(\omega_n \abs{x}^n)
    \end{equation*}
    where $\omega_n$ is the Lebesgue measure of the unit $n$-dimensional ball.
    
    Moreover the \textbf{spherically symmetric increasing rearrangement} of $u$ is defined as
    \begin{equation*}
        u_{\sharp} \colon \R^n \to \R^+ \qquad u_{\sharp} (x) = u_*(\omega_n \abs{x}^n)
    \end{equation*}
\end{definizione}

\noindent It can be proved that
\begin{itemize}

    \item $u^{\sharp}$ ($u_{\sharp}$) is non-negative, radial and radially decreasing (increasing);
    
    \item $u^{\sharp}, u_{\sharp}$ and $u$ are equally distributed which means they have the same distribution function;

    \item the Polya-Szeg\"{o} inequality holds true \cite{Polya_Szego}: if $u \in W_0^{1,p}(\Omega)$, then $u^{\sharp} \in W_0^{1,p}(\Omega^{\sharp})$ and
    \begin{equation*}
        \lVert \nabla u^{\sharp} \rVert_{L^p(\Omega^{\sharp})} \leq \norma{\nabla u}_{L^p(\Omega)}.
    \end{equation*}
    
\end{itemize}
We recall the Theorem of Giarrusso and  Nunziante (\cite[Theorem 2.2]{GN}).
\begin{teorema}
    \label{Giarrusso_Nunziante}
    Let $\Omega \subset \R^n$ be a bounded open set, let $\Omega^{\sharp}$ be the centered ball, let $p \geq 1$, let $f \colon \Omega \to \R$ be a measurable function, let $H \colon \R^n \to \R$ be measurable non-negative functions and let $K \colon [0,+\infty) \to [0,+\infty)$ be a strictly increasing real-valued function such that
    \[
        0 \leq K(\abs{y}) \leq H(y) \qquad \forall y \in \R^n \qquad \text{ and } K^{-1}(f) \in L^p(\Omega).
    \]
    Let $v \in W_0^{1,p}(\Omega)$ be a function that satisfies 
    \[
        \begin{cases}
		H(\nabla v) = f(x) &\text{a.e. in }\Omega \\
		v = 0 &\text{on } \partial \Omega
	\end{cases}
        ,
    \]
    denoting by $z \in W_0^{1,p}(\Omega^{\sharp})$ the unique spherically decreasing symmetric solution to
    \[
	\begin{cases}
		K(\abs{\nabla z}) = f_{\sharp}(x) & \text{a.e. in } \Omega^{\sharp} \\
		z = 0 & \text{on } \partial \Omega^{\sharp}
	\end{cases}
        ,
    \]
    then
    \begin{equation*}
	\norma{v}_{L^1(\Omega)} \leq \norma{z}_{L^1(\Omega^{\sharp})}.
    \end{equation*}
\end{teorema}

Moreover, in \cite{Mercaldo_remark_comparison_Hamilton_Jacobi} the following uniqueness result is proved:
\begin{teorema}
    \label{Mercaldo}
    Let $\Omega \subset \R^n$ be a bounded open set, let $v \in W_0^{1,1}(\Omega)$ be a non-negative function. Denote by $f(x) = \abs{\nabla v}(x)$ and by $w \in W_0^{1,1}(\Omega^{\sharp})$ the decreasing spherically symmetric solution to
    \[
        \abs{\nabla w} = f_{\sharp}.
    \]
    If $\norma{v}_{L^1} = \norma{w}_{L^1}$ then there exists $x_0 \in \R^n$ such that $\Omega = x_0 + \Omega^{\sharp}$, $f=f_{\sharp}(\cdot \, + x_0)$ and $v = w (\cdot \, + x_0)$.
\end{teorema}

From now on $\Omega\subset \R^n$ is a bounded open set with finite perimeter. Let us consider
\[
\text{BV}_0(\Omega):= \left\{u \in \text{BV}(\R^n) :\, u\equiv 0 \text{ in }\R^n \setminus \Omega \right\},
\]
and $u$ a non-negative function belonging to  $\text{BV}_0(\Omega)$.
Let us define
\begin{equation}
    \label{deff_f_piccolo}
    f(x,s) = \bigl( u - u^*(s)\bigr)_+(x) \qquad x \in \R^n, \, s \in [0,+\infty).
\end{equation}

 The function $f(\cdot, s)$ belongs to $\text{BV}_0(\Omega)$ for every $s \in [0, +\infty)$  since it is a truncation of $u$ (See \cite[Theorem 3.96]{Ambrosio_Fusco_Pallara}). Moreover, for every $s \in [0,+\infty)$ we denote by
\begin{equation}
    \label{def_G}
    G(s) = \lvert D f(\cdot ,s) \rvert (\R^n) = \lvert D^a f(\cdot,s) \rvert (\R^n) + \lvert D^s f(\cdot,s) \rvert (\R^n) = G_1(s) + G_2(s),
\end{equation}
where $D^a f$ and $D^s f$ are, respectively, the absolutely continuous part and singular part of the measure $Df$.

The following corollary holds.

\begin{corollario}
\label{fleming_rishel_cor}
    Let $u$ be a non-negative function belonging to  $\text{BV}_0(\Omega)$ and let $G(s)$ be the function defined as in \eqref{def_G}. Then  for a.e. $s \in [0,+\infty)$:
     \begin{equation}
        \label{2.12}
        G(s)= \int_{u^*(s)}^{+\infty} \text{Per}(\{ u>\xi\} ) \, d\xi.
    \end{equation}  
\end{corollario}

\begin{proof}
For a.e. $s \in [0,+\infty)$, applying \ref{fleming_rishel_teo} with $E=\R^n$ to the function $f(\cdot,s)$ defined in \eqref{deff_f_piccolo}, we have
\begin{equation}
   G(s)=\abs{D \bigl( (u-u^*(s) \bigr)_+)}(\R^n) = \int_{-\infty}^{+\infty} \text{Per}\left(\left\{\bigl( u - u^*(s)\bigr)_+> \xi\right\}\right)\, d\xi. 
\end{equation}
Moreover, we have 
\[
    \int_{-\infty}^{+\infty} \text{Per}\left(\left\{\bigl( u - u^*(s)\bigr)_+> \xi\right\}\right)\, d\xi= \int_{0}^{+\infty} \text{Per}\left(\left\{ u - u^*(s)> \xi\right\}\right)\, d\xi,
\]
and a change of variables gives \eqref{2.12}.
\end{proof}

The following properties hold:
\begin{enumerate}

    \item  $G$ is an increasing function on $(0,+\infty)$ by \eqref{2.12}, constant in $(\abs{\Omega}, +\infty)$,  it belongs to $\BV_{\text{loc}}([0,+\infty))$. Then, there exists a positive measure $\sigma$ such that
    \begin{equation}
        \label{deff_dF}
        G(s)= \int_{(0,s]} \, d\sigma(\tau)  \qquad \forall s \in [0,+\infty);
    \end{equation}

    \item $\displaystyle{G_1(s) = \int_{\{u > u^\ast (s)\}} \lvert \nabla^{\mathrm{a}} u \rvert \, dx}$ is increasing and  $AC$ on $[0,+\infty)$, then there exists a function $F_1$  belonging to $L^1([0,+\infty))$:
    \[
    G_1(s)= \int_0^s F_1(\tau)\, d\tau  \qquad \forall s \in [0,+\infty);
    \]

    \item $G_2$ is an increasing function belonging to $\BV_{\text{loc}}([0,+\infty))$, so there exists a positive measure $\sigma_2$ such that
    \[
    G_2(s)= \int_{(0,s]}\, d\sigma_2(\tau)  \qquad \forall s \in [0,+\infty).
    \]
\end{enumerate}

Then, $\forall s\geq 0$
\begin{equation}
    \label{divisione_G}
     G(s) =\sigma((0,s])= \int_{(0,s]} \, d\sigma(\tau)  = \int_{0}^s  F_1(\tau) \,d\tau  + \int_{(0,s]} d\sigma_2(\tau) 
\end{equation}

We will need the following technical lemma which can be proved by arguing as \cite[Lemma 2.1]{Alvino_Trombetti_costanti_maggiorazione}.
\begin{lemma}
    \label{lemma_Alvino_Trombetti}
    Let $\Omega$ be a bounded open set in $\R^n$. If $g \in L^1([0,\abs{\Omega}))$, then there exists a sequence of functions $\{g_k\}$ such that $g_k^* = g^*$ and
    \begin{equation}
        \label{eq_Alvino_Trombetti}
        \lim_k \int_0^{\abs{\Omega}} g_k (s) \varphi (s) \, dx = \int_0^{\abs{\Omega}} g(s) \varphi(s) \, ds, \qquad \forall \varphi \in \BV \bigl( [0,\abs{\Omega}) \bigr).
    \end{equation}
\end{lemma}

\section{Proof of Theorem \ref{Teorema_che_scriveremo}}
\label{sec_proof_of_main_results}
Let us define the following function
\begin{equation}
    \label{deff_v}
    v(s) := \int_s^{+\infty }\frac{1}{ n \omega_n^{\frac{1}{n}} \tau^{1-\frac{1}{n}}} \, d\sigma(\tau)\qquad \forall s \in [0,+\infty),
\end{equation}
where $\sigma$ is defined in \eqref{deff_dF}. We observe that, since $\supp( \sigma) \subseteq [0,\abs{\Omega}]$, $v$ is identically 0 on $(\abs{\Omega}, +\infty)$, hence $v \in \text{BV}_0([0,\abs{\Omega}])$.

As intermediate step towards Theorem \ref{Teorema_che_scriveremo}, we prove the following proposition.

\begin{prop}
    \label{intermedio}
    Let $\Omega \subset \R^n$ be a bounded open set with finite perimeter and assume that $u$ is  a non-negative function belonging to  $\text{BV}_0(\Omega)$. If $v(s)$ is  the function defined as in \eqref{deff_v}, then
    \begin{equation}
        \label{comparison_u*}
        u^*(s) \leq v(s)  \qquad \text{ for a.e. } s \in [0,+\infty).
    \end{equation}
\end{prop}

\begin{proof}
The isoperimetric inequality implies
    \begin{equation*}
        n \omega_n^{\frac{1}{n}} \mu(t)^{1-\frac{1}{n}} \leq \text{Per}(\{u>t\}) \qquad
         \forall t \in [0.+\infty),
    \end{equation*}
     by \eqref{2.12} and \eqref{divisione_G} we have
    \begin{equation*}
        G(s) = \int_{u^*(s)}^{+\infty} \text{Per}(\left\{ u> \xi \right\}) \, d\xi = \int_{(0,s]}\, d\sigma(\tau) \qquad \text{ for a.e. }s \in [0,+\infty).
    \end{equation*}

    Hence, for all $0 \leq s_1 <  s_2  < + \infty$ we have 
    \begin{align*}
    	 \sigma\big((s_1,s_2)\big)=\int_{s_1}^{s_2} \, d\sigma(\tau) &= \lim _{s \to s_2^-}G(s) - G(s_1) \\&=\lim _{s \to s_2^-} \int_{u^\star(s)}^{u^\star(s_1)}\text{Per}(\left\{ u> \xi \right\}) \, d\xi \\& \geq \lim _{s \to s_2^-}\int_{u^\star(s)}^{u^\star(s_1)} n \omega_n^{\frac{1}{n}} \mu(\xi)^{1-\frac{1}{n}} \, d\xi= D\big[H(u^*)\big]\big((s_1,s_2)\big),
    \end{align*}
  where
\[
H(\tau) = \int_{\tau}^{+\infty} n \omega_n^{\frac{1}{n}} \mu(\xi)^{1-\frac{1}{n}} \, d\xi.
\]
    Since this holds for every open interval   $(s_1 ,s_2)$, we have
    \begin{equation}
        \label{confr_measure}
                \sigma(A) \geq D\big[H(u^*)\big](A) \qquad \forall A \subseteq [0, +\infty)\text{ Borel set}.
    \end{equation}

	Observing that $H$ is a Lipschitz function, $D\big[H(u^*)\big]$ is given by (see \cite{ADM})
	  \[
D\big[H(u^*)\big] = 
	\begin{cases}
		-n \omega_n^{\frac{1}{n}} s^{1-\frac{1}{n}} Du^\ast& \text{ on } [0, +\infty) \setminus J_{u^\ast} \\
- n \omega_n^{\frac{1}{n}} s^{1-\frac{1}{n}}\left( (u^*)^+- (u^*)^-\right),& \text{ on }  J_{u^\ast}
	\end{cases}
	\]

   since  $\mu(u^\ast(s))= s$ a.e. with respect $Du^\ast$ (by the properties of the rearrangements) and since for $s \in J_{u^\ast}$

    \begin{align*}
        H\bigl( ( (u^*)^+(s) \bigr)- H \bigl( ( (u^*)^-(s)\bigr) &=  \int_{ u^*(s)}^{u^*(s^-)} n \omega_n^{\frac{1}{n}} \mu(\xi)^{1-\frac{1}{n}} \, d\xi \\
        &= - n \omega_n^{\frac{1}{n}} s^{1-\frac{1}{n}}\left( (u^*)^+(s)- (u^*)^-(s)\right).
    \end{align*}
    Then we can write 
    \begin{equation}
        \label{DH}
        \frac{dD\big[H(u^*)\big] }{dDu^\ast}=  - n \omega_n^{\frac{1}{n}} s^{1-\frac{1}{n}}.
    \end{equation}
     Therefore, by means of \eqref{confr_measure}, \eqref{DH}, we have 
    \begin{equation*}
         u^*(s)  = -\int_s^{+\infty} \, d(D u^\ast)(\tau) =  \int_s^{+\infty} \frac{dD\big[H(u^*)\big](\tau)}{n \omega_n^{ \frac{1}{n} }\tau^{1 - \frac{1}{n}} }\leq \int_s^{+\infty} \frac{d\sigma(\tau)}{n \omega_n^{ \frac{1}{n} } \tau^{1 - \frac{1}{n}} }= v(s).
    \end{equation*}
\end{proof}

\vspace{1 em}

Now we are in position to prove the main theorem.

\begin{proof}[Proof of Theorem \ref{Teorema_che_scriveremo}]
    First of all, let us emphasize that the decreasing rearrangement of $u^\star$, defined in \eqref{eq_che_risolve_u_picche}, is
     \[
        (u^\star)^\ast(s)= \int_s^{+\infty} \frac{ \lvert \nabla^{\mathrm{a}} u \rvert_* (t) }{n\omega_n^{\frac{1}{n}}t^{1-\frac{1}{n}}} \, dt  + \frac{1}{\text{Per}(\Omega^\sharp)} \abs{D^s u}(\R^n) \, \chi_{[0,\abs{\Omega}]}(s) \qquad \forall s \in [0,+\infty).
    \]
    Now, let us integrate \eqref{comparison_u*} between $0$ and $+\infty$ and let us use Fubini's Theorem to obtain
    \begin{align*}
       \int_0^{+\infty} u^*(s) \, ds & \leq \int_0^{+\infty} v(s) \, ds \\
        & = \frac{1}{n \omega_n^{\frac{1}{n}}} \int_0^{+\infty}  \Biggl( \int_s^{+\infty} \, \frac{d\sigma(t)}{t^{1-\frac{1}{n}}} \Biggr) \, ds \\[1ex]
        & = \frac{1}{n\omega_n^{\frac{1}{n}}} \int_0^{+\infty} \Biggl( \int_0^t \, \frac{ds}{t^{1-\frac{1}{n}}} \Biggr) \, d\sigma(t) \\
        & = \frac{1}{n\omega_n^{\frac{1}{n}}} \int_0^{+\infty} t^{\frac{1}{n}} \, dF(t)\\
        &=\frac{1}{n\omega_n^{\frac{1}{n}}} \left[\int_0^{+\infty}t^{\frac{1}{n}} F_1(t) \, dt + \int_0^{+\infty} t^{\frac{1}{n}} \, d\sigma_2(t)\right].
    \end{align*}
    By \eqref{eq_Alvino_Trombetti} applied to $F_1$ and the Hardy-Littlewood inequality \eqref{Hardy-Littlewood}, we have
    \begin{align*}
         \int_0^{+\infty} t^{\frac{1}{n}} F_1(t) \, dt &=\int_0^{\abs{\Omega}} t^{\frac{1}{n}} F_1(t) \, dt = \lim_k \int_0^{\abs{\Omega}} t^{\frac{1}{n}} (F_1)_k(t) \, dt \\
         &\leq \int_0^{\abs{\Omega}} t^{\frac{1}{n}} \lvert \nabla^{\mathrm{a}} u \rvert_* (t) \, dt = \int_0^{+\infty} t^{\frac{1}{n}} \lvert \nabla^{\mathrm{a}} u \rvert_* (t) \, dt,
    \end{align*}
    then
    \begin{equation}
    \label{stima_cruciale}
        \begin{aligned}
        \int_0^{+\infty} u^*(s) \, ds & \leq \frac{1}{n\omega_n^{\frac{1}{n}}} \Biggl[ \int_0^{+\infty}  t^{\frac{1}{n}} \lvert \nabla^{\mathrm{a}} u \rvert_* (t) \, dt +  \int_0^{+\infty} t^{\frac{1}{n}} \, d\sigma_2(t) \Biggr] \\
        & \leq \frac{1}{n\omega_n^{\frac{1}{n}}} \Biggl[ \int_0^{+\infty}  t^{\frac{1}{n}} \lvert \nabla^{\mathrm{a}} u \rvert_* (t) \, dt + \abs{\Omega}^{\frac{1}{n}}  \int_0^{+\infty} \, d\sigma_2(t) \Biggr],
    \end{aligned}
    \end{equation}
    since $F_2(A)=0$ for all $A \subset (\abs{\Omega},+\infty)$.

    Using again Fubini's Theorem, we can compute
    \begin{align*}
        \int_0^{+\infty} \lvert \nabla^{\mathrm{a}} u \rvert_* (t) t^{\frac{1}{n}} \, dt  & = \int_0^{+\infty} \frac{ \lvert \nabla^{\mathrm{a}} u \rvert_* (t) }{t^{1-\frac{1}{n}}} \int_0^t \, ds = \int_0^{+\infty} \Biggl( \int_s^{+\infty} \frac{ \lvert \nabla^{\mathrm{a}} u \rvert_* (t) }{t^{1-\frac{1}{n}}}  \, dt \Biggr) \, ds,
    \end{align*}
    and
    \begin{equation*}
        \frac{\abs{\Omega}^{\frac{1}{n}}}{n \omega_n^\frac{1}{n}} \int_0^{+\infty} dF_2(t) = \abs{\Omega} \frac{1}{\text{Per}(\Omega^\sharp)} \abs{D^s u}(\R^n) = \int_0^{+\infty}\frac{1}{\text{Per}(\Omega^\sharp)} \abs{D^s u}(\R^n) \chi_{[0,\abs{\Omega}]}(s) \,ds.
    \end{equation*}
    
    Hence, \eqref{stima_cruciale} can be written as
    \begin{align*}
        \norma{u}_{L^1(\Omega)} \leq \int_0^{+\infty}\Biggl[ \int_s^{+\infty} \frac{ \lvert \nabla^{\mathrm{a}} u \rvert_* (t) }{n\omega_n^{\frac{1}{n}}t^{1-\frac{1}{n}}} \, dt  + \frac{1}{\text{Per}(\Omega^\sharp)} \abs{D^s u}(\R^n) \chi_{[0,\abs{\Omega}]}(s)\Biggr]\, ds =\norma{u^\star}_{L^1(\Omega^{\sharp})}.
    \end{align*}    
\end{proof}

\begin{oss}
	We stress the following facts:
    \[
        \abs{D^{\mathrm{a}} u}(\R^n) = \int_{\R^n} \abs{\nabla^\mathrm{a} u} \, dx = \int_{\Omega^{\sharp}} \abs{\nabla^\mathrm{a} u^{\star}} \, dx \quad\text{ and } \quad\abs{D^{s} u}(\R^n) = \abs{D^{s} u^\star}(\R^n),
    \]
    and then 
     \[
            \abs{Du}(\R^n)=\abs{Du^\star}(\R^n).
        \]
\end{oss}

\section{Two versions of the torsional rigidity}
\label{sec_some_applications}

For a given $\Lambda >0$ we consider
\begin{equation}
    \label{deff_funzionale_F}
    \F_{\Lambda}(\psi) := \frac{1}{2} \int_{\Omega} \abs{\nabla \psi}^2 \, dx - \int_{\Omega} \psi \, dx + \Lambda\abs{ \left\{ \abs{\nabla \psi} \neq 0 \right\}} \qquad \psi \in H_0^1(\Omega),
\end{equation}
 and the associated minimum problem:
\begin{equation}
    \label{penalized_torsional_rigidity}
    T_\mathcal{F}(\Omega,\Lambda) := - \inf_{\psi \in H_0^1(\Omega)} \F_{\Lambda}(\psi).
\end{equation}
First of all, let us observe that the minimum can be found among non-negative functions. Indeed, passing from $\psi$ to $\abs{\psi}$ it holds $\mathcal{F}(\psi) \geq \mathcal{F}(\abs{\psi})$.

Assuming that problem \eqref{penalized_torsional_rigidity} admits a minimum $u \in H^1_0(\Omega)$, then it is also a maximum for the torsional rigidity defined by Diaz, Polya and  Weinstein in \cite{Torsional_Rigidity_Diaz_Weinstein,Torsional_rigidity_Polya_Weinstein} of a multiply-connected cross-section with fixed measure of the holes, that is
\[
    T(\Omega) = \max_{\substack{ \psi \in C_0(D) \cap C^1(\Omega) \\ \psi \text{ constant} \\ \text{in every }A_i }} \frac{\displaystyle{\biggl(\int_{D} \psi \, dx \biggr)^2}}{\displaystyle{\int_{D} \abs{\nabla \psi}^2 \, dx}},
\]
where $A_i$ are the connected component of $\{\abs{\nabla u }=0\}$ and $D = \Omega \cup \bigcup_i A_i$.

Functionals with penalizing terms are very common in the mathematical modelling of physical problems.
 The bibliography is very wide and some cornerstones are \cite{Alt_Caffarelli,De_Giorgi_Carriero_Leaci}.

However, in the literature,  penalizing terms of the form $\abs{\left\{\abs{\nabla \psi} \neq 0\right\}}$ are quite unusual.
The main difficulty in the study of  \eqref{penalized_torsional_rigidity} is to prove the existence of a minimizer because of the lack of the  lower semicontinuity of the functional.

For this reason, we prove the existence of a minimizer in the case when $\Omega$ is a ball.

\begin{prop}
    \label{teo_esistenza}
    Let $\Lambda, R>0$ and let $B_R$ be the  centered ball with radius $R$. Then the functional $\mathcal{F}_\Lambda$ defined in \eqref{deff_funzionale_F} admits a minimizer $v$  belonging to $H_0^1(\Omega)$. Such a minimizer is unique up to a sign, it is radially symmetric and $\abs{\nabla v}$ is radially increasing.
\end{prop}

\begin{proof}
	We divide the proof in 3 steps.
	
	\begin{enumerate}
		\item Boundness from below.
  
            \noindent First of all, let us prove that the functional $\F_{\Lambda}$ is bounded from below for every choice of $\Lambda$ and for every $R>0$. For all $\psi \in H^1_0(B_R)$, sing Young and Poincaré inequalities, we get
		\begin{align*}
                \F_{\Lambda}(\psi) &= \frac{1}{2} \int_{B_R} \abs{\nabla \psi}^2 \, dx - \int_{B_R} \psi \, dx + \Lambda\abs{ \left\{ \nabla \psi \neq 0 \right\}} \\
                & \geq \frac{1}{2} \int_{B_R} \abs{\nabla \psi}^2 \, dx - \eps \int_{B_R} \frac{\psi^2}{2} - \frac{\abs{B_R}}{2\eps} \\
                & \geq \frac{1}{2} \int_{B_R} \abs{\nabla \psi}^2 \, dx - \frac{\eps C(n,B_R)}{2} \int_{B_R} \abs{\nabla \psi}^2 \, dx  - \frac{\abs{B_R}}{2\eps} \\
                &= \frac{(1-\eps C(n,B_R))}{2} \int_{B_R} \abs{\nabla \psi}^2 \, dx - \frac{\abs{B_R}}{2 \eps}.
		\end{align*}
		Chosing $\eps$ sufficiently small such that
		\[
		  0< \eps \leq \frac{1}{C(n,B_R)}
		\]
		then
		\[
			\F_{\Lambda}(\psi) \geq - \frac{\abs{B_R} }{2 C(n,B_r)}
			\geq - C (n,B_R) > -\infty
		\]
		so
		\[
		      T(B_R, \Lambda) = -\inf_{\psi \in H_0^1(B_R)} \F_{\Lambda}(\psi) <\infty.
		\]

		\item Compactness and semicontinuity. 
            
            \noindent Now we consider a minimizing sequence $\{\psi_k\}$ for $T_\mathcal{F}(B_R,\Lambda)$ and we prove that it is bounded in $H_0^1(B_R)$. We can assume that $\F_{\Lambda}(\psi_k) \leq -T_\mathcal{F}(B_R,\Lambda)+1$ and by Proposition \ref{Giarrusso_Nunziante} we can assume that $\psi_k$ are radial function with $\abs{\nabla \psi_k}$ radially symmetric increasing.
		
		\noindent Using Young and Poincaré inequalities, we obtain
		\begin{align*}
                \F_{\Lambda}(\psi_k) & = \frac{1}{2} \int_{B_R} \abs{\nabla \psi_k}^2 \, dx - \int_{B_R} \psi_k \, dx + \Lambda\abs{ \left\{ \nabla \psi_k \neq 0 \right\}} \\
                & \geq \frac{1}{2} \int_{B_R} \abs{\nabla \psi_k}^2 \, dx - \int_{B_R} \psi_k \, dx \\
                & \geq \frac{1}{2} \int_{B_R} \abs{\nabla \psi_k}^2 \, dx - \eps \int_{B_R} \frac{\psi_k^2}{2} - \frac{\abs{B_R}}{2\eps}  \\
                & \geq \frac{1}{2} \int_{B_R} \abs{\nabla \psi_k}^2 \, dx - \frac{\eps C(n,B_r)}{2} \int_{B_R} \abs{\nabla \psi_k}^2 \, dx  - \frac{\abs{B_R}}{2\eps}   \\
                &= \frac{1-\eps C(n,B_R)}{2} \int_{B_R} \abs{\nabla \psi_k}^2 \, dx - \frac{\abs{B_R}}{2\eps}.
		\end{align*}
		
		Choosing $\eps <\displaystyle{ \frac{1}{C(n,B_R) }}$ we have
		\[
                -T_\mathcal{F}(B_R,\Lambda)+1 \geq \F_{\Lambda}(\psi_k) \geq \frac{1}{4} \int_{B_R} \abs{\nabla \psi_k}^2 \, dx - C (B_R)
		\]
            then by Poincaré inequality, the sequence $\{\psi_k\}$ is bounded in $H_0^1(B_R)$.
            
            \noindent This implies that there exists a subsequence (still denoted by $\psi_k$) and a function $v \in H_0^1(B_R)$ such that $\psi_k \to v$ strongly in $L^2(\Omega)$, a.e. in $\Omega$ and $\nabla \psi_k \rightharpoonup \nabla v$ weakly in $L^2$. Let us show that $v$ is a minimum for $\F_{\Lambda}$.
		
		\noindent The lower semicontinuity of the norms gives
		\begin{equation}
			\label{semicontinuità_primi_2_termini}
                \liminf_k \biggl[ \frac{1}{2} \int_{B_R} \abs{\nabla \psi_k}^2 \, dx - \int_{B_R} \psi_k \, dx \biggr]  \geq \frac{1}{2} \int_{B_R} \abs{\nabla v}^2 \, dx - \int_{B_R} v \, dx.
		\end{equation}
  
            \noindent Let us deal with the last term of $\F_{\Lambda}$ and let us prove that
		\begin{equation*}
                \liminf_k \abs{ \left\{\abs{\nabla u_k}  \neq 0 \right\} } \geq \abs{\left\{ \abs{\nabla u} \neq 0 \right\} }.
		\end{equation*}
            Denoting by $r_k$ the radius of the ball where $\abs{\nabla \psi_k}=0$, we can assume that $r_k$ converges to some $r \geq 0$. Therefore
		\[
                \liminf_k \, \abs{ \left\{\abs{\nabla \psi_k} \neq 0\right\}} = \lim_k \,  [\omega_n (R^n-r_k^n)] = \omega_n (R^n - r^n ).
		\]
            So we have just to prove that $\abs{\nabla v} = 0$ in $B_r$. Since $\{ \psi_k \}$ are radial functions, obviously $v$ is radial  too.
		
		If $r=0$ there is nothing to prove.
  
            If $r > 0$, assume by contradiction that there exists $A \subset B_r$ with $\abs{A}>0$ and that $\abs{\nabla v} \neq 0$ in $A$. Clearly there exists $\eps>0$ such that $\abs{A \cap B_{r-\eps}} >0 $.
            
            \noindent Since $r_k \to r$ if we choose a function $g \in C^{\infty}_C(B_R, \R^n)$ with support included in $A \cap B_{r-\eps}$ we have
		\[
                \int_{B_R} \langle \nabla v , g \rangle \, dx = \lim_k \int_{B_R} \langle \nabla \psi_k , g \rangle \, dx = 0.
		\]
            Since this must be true for every $g \in C^{\infty}_{C}(A \cap B_{r-\eps},\R^n)$, we get a contradiction.

		\noindent Then in any case
		\begin{equation}
			\label{semicontinuity_of_support_measure}
                \liminf_k \abs{\left\{ \abs{\nabla \psi_k} \neq 0\right\}} \geq \abs{\left\{ \abs{\nabla v} \neq 0\right\}}.
		\end{equation}
  
            \noindent By \eqref{semicontinuità_primi_2_termini} and\eqref{semicontinuity_of_support_measure}, we get
		\[
                -T_\mathcal{F}(B_R, \Lambda)=\liminf_k \F_{\Lambda}(\psi_k) \geq \F_{\Lambda}(v) \geq -T_\mathcal{F}(B_R,\Lambda)
		\]
		so $v$ is a minimum of $\F_{\Lambda}$ in $B_R$.

            \item Uniqueness.
            
            \noindent Let us suppose that $v$ is a minimum of $\mathcal{F}_\Lambda(\psi)$. By Theorem \ref{Giarrusso_Nunziante}, it exists $\overline{v} \in H_0^1(B_R)$ such that
            \[
                \F_{\Lambda} (v) \geq \F_{\Lambda}(\overline{v})
            \]
            and since $v$ is minimum, it holds
            \[
                \F_{\Lambda} (v) = \F_{\Lambda}(\overline{v}).
            \]
            Since $\abs{\nabla v}$ is equally distributed with $\abs{\nabla \overline{v}}$, the previous equality implies
            \[
                \norma{v}_{L^1} = \norma{\overline{v}}_{L^1}
            \]
            so Theorem \ref{Mercaldo} gives that $\abs{v} = \overline{v}$.

	\end{enumerate}
\end{proof}

\begin{oss}
    \label{osse_per_minimo}
    We highlight that Theorem \ref{Giarrusso_Nunziante} ensures us that the minimum when $\Omega$ is a ball has gradient equal to zero only in a ball $B_r$ centered at the origin with $0 \leq r \leq R.$
\end{oss}

Now, as already mention in the introduction, we prove a Saint-Venant type inequality for $T_\mathcal{F}(\Omega,\Lambda)$.
\begin{corollario}
    \label{teorema_palla_meglio}
    Let $\Omega \subset \R^n$ be a bounded open set with finite perimeter and let $\Omega^{\sharp}$ be the centered ball. If $\Lambda > 0$, then
    \begin{equation*}
        T_\mathcal{F}(\Omega,\Lambda) \leq T_\mathcal{F}(\Omega^{\sharp}, \Lambda).
    \end{equation*}
\end{corollario}
    \begin{proof}
    For every function $\psi \in H_0^{1}(\Omega)$, by Theorem \ref{Giarrusso_Nunziante} or \ref{Teorema_che_scriveremo}, there exists $\ov{\psi} \in H_0^{1}(\Omega^{\sharp})$ that satisfies
    \[
        \F_{\Lambda}(\psi) \geq \F_{\Lambda}(\ov{\psi}) \geq -T_\mathcal{F}(\Omega^{\sharp},\Lambda)
    \]
    and then
    \[
        T_\mathcal{F}(\Omega, \Lambda) \leq T_\mathcal{F}(\Omega^{\sharp},\Lambda).
    \]
\end{proof}

Now we deal with the functional 
\begin{equation*}
    \mathcal{G}(\psi) := \frac{\displaystyle{\int_{\Omega} \abs{\nabla \psi}^2 \, dx + \frac{1}{m}\left( \int_{\partial \Omega} \abs{\psi} \, d\mathcal{H}^{n-1}\right)^2}}{\displaystyle{\left(\int_{\Omega} \abs{\psi} \, dx\right)^2}}  \qquad \psi \in H^1(\Omega).
\end{equation*}
with $m > 0$.

The interest in this type of functional is related to the problem of optimal insulation in a given domain. Indeed, the minimum of $\mathcal{G}$ gives the long-time distribution of temperature of the domain $\Omega$ and the displacement around $\Omega$ of a thin layer of insulator with total mass equal to $m$. We refer to \cite{Bucur_Buttazzo_Nitsch} for more details.

If $\Omega$ is a Lipschitz domain, $\mathcal{G}(\psi)$ achieves its minimum among all $H^1(\Omega)$ functions. So we define
\[
    \frac{1}{T_\mathcal{G}(\Omega,m)}:= \min_{\psi \in H^1(\Omega)} \mathcal{G}(\psi).
\]

It is easy to check that the Euler–Lagrange equation of this functional is
\begin{equation*}
    \begin{cases}
         -\Delta u= 1 & \text{in }\Omega \\
        \displaystyle{\frac{\partial u}{\partial \nu} + \frac{1}{m} \int_{\partial\Omega} \abs{u} \, d\mathcal{H}^{n-1}} = 0 &\text{ on } \partial \Omega.
    \end{cases}
\end{equation*}

So Theorem \ref{Teorema_che_scriveremo} gives us the following Saint-Venant type inequality for $T_{\mathcal{G}}(\Omega)$.
\begin{corollario}
    \label{teorema_palla_meglio2}
    Let $\Omega \subset \R^n$ be a bounded open set with finite perimeter and let $\Omega^{\sharp}$ be the centered ball. If $m > 0$, then
	\begin{equation*}
		T_\mathcal{G}(\Omega,m) \leq T_\mathcal{G}(\Omega^{\sharp},m).
	\end{equation*}
\end{corollario}

\begin{proof}
    For every function $\psi \in H^{1}(\Omega)$, by \ref{Teorema_che_scriveremo}, there exists $\ov{\psi} \in H^{1}(\Omega^{\sharp})$ that satisfies
    \[
        \mathcal{G}(\psi) \geq \mathcal{G}(\ov{\psi}) \geq \frac{1}{T_\mathcal{G}(\Omega^{\sharp},m)}
    \]
    and then
    \[
        T_\mathcal{G}(\Omega,m) \leq T_\mathcal{G}(\Omega^{\sharp},m).
    \]
\end{proof}

\noindent \textbf{Competing Interests} The authors have no competing interests as defined by Springer, or other interests that might be perceived to influence the results and/or discussion reported in this paper.

\vspace{0.3 cm}

\noindent \textbf{Funding Information} This work has been partially supported by the PRIN project (Italy) Grant: “Direct and inverse problems for partial differential equations: theoretical aspects and applications” and by GNAMPA of INdAM.

 \vspace{0.3 cm}

\noindent \textbf {Author contribution} All authors  equally  contributed to  the paper.

 \vspace{0.3 cm}
 
\noindent \textbf{Acknowledgements} We would like to thank the anonymous Referees for their valuable comments and suggestions.

\addcontentsline{toc}{chapter}{Bibliografia}

\printbibliography[heading=bibintoc, title={References}]

\end{document}